\documentclass{amsart}

\usepackage{graphicx}
\usepackage{amssymb}
\usepackage{mathrsfs}
\usepackage{amsmath}
\usepackage{amsthm}
\usepackage{braket}
\usepackage{tikz-cd}
\tikzset{
    labl/.style={anchor=south, rotate=90, inner sep=.5mm}
}
\usepackage{stmaryrd}
\usepackage{hyperref}
\usetikzlibrary{positioning}
\usepackage{cleveref}
\usepackage[numbers]{natbib}

\newtheorem{theorem}{Theorem}[section]
\newtheorem{proposition}[theorem]{Proposition}
\newtheorem{lemma}[theorem]{Lemma}
\newtheorem{cor}[theorem]{Corollary}
\newtheorem{conj}[theorem]{Conjecture}

\theoremstyle{definition}
\newtheorem{definition}[theorem]{Definition}

\newtheorem{question}[theorem]{Question}

\newtheorem{remark}[theorem]{Remark}

\def\gl{\mathop{\mathrm{Gl}}\nolimits}

\def\pr{\mathop{\mathrm{pr}}\nolimits}

\def\spec{\mathop{\mathrm{Spec}}\nolimits}
\def\id{\mathop{\mathrm{id}}\nolimits}
\usepackage[scr=boondoxo]{mathalfa}
\newcommand{\kk}{\mathop{\mathscr{k}}\nolimits}
\newcommand{\Sym}{\mathop{\mathrm{Sym}}\nolimits}
\newcommand{\Supp}{\mathop{\mathrm{Supp}}\nolimits}

\title[The Hitchin morphism for $K$-trivial varieties]{The Hitchin morphism for $K$-trivial varieties}
\author{Aryaman Patel and Dario Weissmann}
\date{\today}

\begin{document}

\begin{abstract}
    We study the Hitchin morphism for higher dimensional varieties and show that, for a certain class of varieties which we call \emph{$r$-small}, the set-theoretic image of the Hitchin morphism from the Dolbeault moduli space coincides with the spectral base. In other words, a stronger version of the conjecture of Chen and Ng\^o holds for this class of varieties, which includes $K$-trivial varieties.
    As part of the proof,
    we slightly modify the construction of spectral covers to obtain normal spectral covers. 
\end{abstract}

\maketitle

\section{Introduction}
  
We work over an algebraically closed field $\kk$ of characteristic zero.
For a smooth proper scheme $X$ over $\kk$, 
the Hitchin morphism is a map from the moduli stack $\mathscr{M}^r_X$ 
of Higgs bundles of rank $r$ to the \emph{Hitchin base} 
$\mathscr{A}^r_X:=\bigoplus_{i=1}^rH^0(X,\mathrm{Sym}^i\Omega^1_X)$,
which sends a Higgs bundle $(\mathcal{E},\theta)$ to the characteristic
polynomial of the Higgs field $\theta$.
We denote it by $h^r_X:\mathscr{M}^r_X\to\mathscr{A}^r_X$.

When $X$ is a curve, the Hitchin morphism is known to be surjective. Indeed, in this case $\Omega^1_X$ is a line bundle, so the integrability condition $\theta\wedge\theta=0$ is automatically satisfied by any vector bundle with a $\Omega^1_X$-twisted endomorphism. However, for higher dimensional $X$, the Hitchin morphism is in general far from being surjective. A natural question then is: \emph{can we describe the image of $h^r_X$ when $\mathrm{dim}X>1$?}

Recent progress has been made in this direction by Chen and Ng\^o in \cite{CN20}, where they show that $h^r_X$ factors through a closed subscheme $\mathscr{B}^r_X$ of $\mathscr{A}^r_X$, called the \emph{spectral base}, which is in general a non-linear subspace of smaller dimension. The morphism $\mathscr{M}^r_X\to\mathscr{B}^r_X$ is called the \emph{spectral data morphism} and is conjectured to be surjective. 

\begin{conj}[{\cite[Conjecture 5.2]{CN20}}]\label{conjecture}
For every point $s\in\mathscr{B}^r_X$, the fiber $(h^r_X)^{-1}(s)$ is non-empty.
\end{conj}

Although the conjecture has been formulated more generally for $G$-Higgs bundles where $G$ is a reductive group, we only treat the case $G=GL_r$ in this article. Some recent progress has been made towards proving Conjecture \ref{conjecture}. For example, Chen and Ng\^o proved it for ruled and elliptic surfaces in \cite{CN20}, and Song-Sun proved it more generally for all surfaces in \cite{SS24}.
The case of hyperelliptic varieties was treated by the authors
in \cite{PW26}.
In \cite{HL24} He and Liu proved a stronger version of the conjecture for rank $2$ Higgs bundles,
namely that the Hitchin morphism from the moduli stack of polystable rank $2$ Higgs bundles surjects onto the spectral base $\mathscr{B}^2_X$.
In \cite{HLM24} He, Liu, and Mok show that $\mathscr{B}^r_X=\{0\}$ for all $r$ when $X$ is a smooth projective quotient of a bounded symmetric domain. In particular, Conjecture \ref{conjecture} holds for such varieties.

The motivating question for the present article is:
for which projective varieties $X$ is the Hitchin morphism
from the moduli space $M^{P,G-ss}_X$ of Gieseker semistable
Higgs bundles with fixed Hilbert polynomial $P$
surjective onto the spectral base?
A natural choice for $P$ is the Hilbert polynomial $P_0$
of the trivial bundle of rank $r$,
in which case $M^{P_0,G-ss}_X=:M^r_{X,\mathrm{Dol}}$
is the \emph{Dolbeault moduli space}.
The following result partially answers this question.

\begin{theorem}[\Cref{ktrivial}]\label{thm-ktrivial}
    Let $X$ be a smooth projective variety whose canonical divisor
    $K_X$ is numerically trivial.
    Then for all $r\in\mathbb{N}$ the restricted Hitchin morphism
    $h^r_{X,\mathrm{Dol}}:M^r_{X,\mathrm{Dol}}\to\mathscr{B}^r_X$
    is surjective.
\end{theorem}

The surjectivity of $h^r_{X,\mathrm{Dol}}$ actually holds
for a larger class of varieties that satisfy a geometric condition
on the vanishing loci of symmetric differentials. 
Theorem \ref{thm-ktrivial} is thus a special case
of the following more general statement.

\begin{theorem}[\Cref{generalprop}]\label{generalthm}
    Let $X$ be a smooth projective scheme and 
    let $r\in\mathbb{N}$ be fixed.
    Suppose that $\mathrm{Sym}^{k(k-1)}\Omega^1_X$
    does not have a non-zero global section whose vanishing
    locus has codimension $1$ in $X$, for all $1\le k\le r$.
    Then the restricted Hitchin morphism
    $h^k_{X,\mathrm{Dol}}:M^k_{X,\mathrm{Dol}}\to\mathscr{B}^k_X$
    is surjective for all $1\le k\le r$.
\end{theorem}

In fact, a stronger statement holds -- every fiber of 
$h^k_{X,\mathrm{Dol}}$ contains a Higgs bundle 
whose underlying vector bundle is semistable and 
has vanishing Chern classes 
(see Proposition \ref{generalprop} and Remark \ref{ettriv}).  
The proof relies essentially on a modification
of the spectral cover construction and purity of branch locus,
which is recalled in the next section.
Although many varieties satisfy the hypothesis of Theorem \ref{generalthm},
the condition that all non-zero global sections of 
$\mathrm{Sym}^{r(r-1)}\Omega^1_X$ have vanishing locus 
of codimension $\ge2$ in $X$ is hard to check.
This motivates the notion of \emph{$r$-small varieties}.

\begin{definition}\label{rsmall}
    Let $X$ be a smooth projective variety and
    let $r\in\mathbb{N}$ be fixed.
    Then $X$ is said to be \emph{$r$-small} 
    if $\Omega^1_X$ is slope-semistable and 
    $\mathrm{Sym}^{k(k-1)}\Omega^1_X$ satisfies 
    the slope inequality $\mu(\mathrm{Sym}^{k(k-1)}\Omega^1_X)<1$
    for all $1\le k\le r$. 
\end{definition}

\begin{remark}
    Note that in characteristic $0$, a vector bundle $\mathcal{V}$ being
    semistable implies that the bundle $\mathrm{Sym}^j\mathcal{V}$ 
    is semistable for all $j\ge0$.
\end{remark}

A standard argument shows that any $r$-small variety
satisfies the hypothesis of Theorem \ref{generalthm}.
It follows that the Hitchin morphism
$h^r_{X,\mathrm{Dol}}:M^r_{X,\mathrm{Dol}}\to\mathscr{B}^r_X$
is surjective for any $r$-small variety $X$.
Since $\Omega^1_X$ is known to be slope-semistable of slope $0$ 
for any 
projective variety $X$ with numerically trivial canonical divisor $K_X$ (e.g. by \cite[Theorem A]{G16}),
it is immediate that any such variety is $r$-small.

In general, the restricted Hitchin morphism $h^r_{X,\mathrm{Dol}}$ 
from the Dolbeault moduli space is not surjective 
(see \cite[Example 3.4 (5)]{HL24}).
It would then be interesting to determine
necessary and sufficient conditions
that $X$ must satisfy so that $h^r_{X,\mathrm{Dol}}$ is surjective.
We believe that this question can be answered
once we have more information on the
geometric and topological properties of the spectral base.

In \cite{SS24}, Song-Sun prove that the spectral base is invariant under smooth birational modifications. We use this along with the standard argument of resolving the singularities
of a torsion-free sheaf to reformulate
\Cref{conjecture} as follows.

\begin{conj}[equivalent to Conjecture \ref{conjecture}]
    The set-theoretic image $\mathrm{im}(h^r_X)$ of the Hitchin morphism is invariant under smooth birational modifications.
\end{conj}

\section*{Acknowledgements}
AP acknowledges support from the Deutsche Forschungsgemeinschaft (DFG, German Research Foundation) -- Project ID 286237555 (TRR 195) and Project ID 530132094. He also thanks Vladimir Lazi\'c for helpful comments on an earlier draft of this article.

\section*{Notation}
We work over an algebraically closed base field $\kk$ of characteristic $0$.
By \emph{variety} we mean an integral separated $\kk$-scheme of finite type.
Given an affine group scheme $G$ over $\kk$ acting on an affine scheme
$A$, we denote the quotient stack by $[A/G]$ which over a scheme $S$
parametrizes $G$-torsors over $S$
together with a $G$-invariant morphism to $A$.
We denote by $\mathrm{B}G$ the quotient stack $[\spec(\kk)/G]$,
where we equip the point with the trivial action.

For a projective variety $X$ we implicitly fix a polarization
$\mathcal{O}_X(1)$ and define the notion
of Gieseker (semi)stability, slope-(semi)stability, Hilbert polynomial,
and degree with respect to $\mathcal{O}_X(1)$.

\section{Spectral covers}\label{sec-spectral covers}
In this section, we recall the definitions of the Hitchin base,
the spectral base, 
and the construction of spectral covers from \cite[Section 6]{CN20}.
We also define 
a stratification of the spectral base in terms of multiplication
of characteristic polynomials, similar to the stratification
on the Chow scheme defined in \cite{SS24},
and analogous to the stratification of the Hitchin base defined in \cite{PW26}.

The spectral covers arising from
the open locus of irreducible spectral data are close to being integral,
see \Cref{lemma-generically-integral} for the precise formulation.
Further, we recall the discriminant and show that
we can factor every spectral datum into a product of irreducible 
spectral data with non-zero discriminant, see \Cref{zerodisc}.

\subsection{The spectral base}
In this subsection, we recall the definitions of the Hitchin base,
the spectral base, and the construction of spectral covers from 
\cite[Section 6]{CN20}.
We refer the reader to loc. cit. for a complete and rigorous treatment.

Let $X$ be a smooth variety of dimension $d$,
and let $r\in\mathbb{N}$ be fixed.
Define $A^r$ as the affine space
$\mathbb{A}^d\times\mathrm{Sym}^2(\mathbb{A}^d)\times\dots\times
\mathrm{Sym}^r(\mathbb{A}^d)$
and $B^r$ as the Chow scheme $\mathrm{Chow}^r(\mathbb{A}^d):=(\mathbb{A}^d)^r\sslash S_r$,
which classifies $0$-dimensional cycles of length $r$ in $\mathbb{A}^d$.
Here, $S_r$ denotes the symmetric group in $r$ elements
and acts via permutation. 

A closed point $\omega\in B^r$
can be represented as an unordered tuple of $r$ closed points 
$\omega=[\omega_1,...,\omega_r]$, 
where $\omega_i\in\mathbb{A}^d$ for $1\le i\le r$.
We denote the $i$-th elementary symmetric polynomial 
in $\omega_1,\dots,\omega_r$ 
by $\sigma_i(\omega)=\sigma_i(\omega_1,\dots,\omega_r)
\in\mathrm{Sym}^i(\mathbb{A}^d)$. 
For example, $\sigma_1(\omega)=\omega_1+\dots+\omega_r\in\mathbb{A}^d$,
$\sigma_2(\omega)=\sum_{i<j}\omega_i\cdot\omega_j\in\mathrm{Sym}^2
(\mathbb{A}^d)$,
and $\sigma_r(\omega)=
\omega_1\cdot\omega_2\cdots\omega_r\in\mathrm{Sym}^r(\mathbb{A}^d)$.
By \cite[Theorem 4.1]{CN20} the morphism
\begin{align}\label{closed-embedding}
     B^r &\to A^r\\
     [\omega_1,\dots,\omega_r] &\mapsto
     (\sigma_1(\omega),\dots,\sigma_r(\omega))\nonumber
\end{align}
is a closed immersion. 

The natural $\gl_d$-action on $\mathbb{A}^d$
induces compatible $\gl_d$-actions on $A^r$ and $B^r$.
Consider the quotient stacks $\mathscr{A}^r:=[A^r/\gl_d]$ 
and $\mathscr{B}^r:=[B/\gl_d]$.
Then we define the Hitchin base
$\mathscr{A}^r_X:=\mathscr{H}om(X,\mathscr{A}^r)
\times_{\mathscr{H}om(X,\mathrm{B}\gl_d)}\mathrm{Spec}(\kk)$
and the spectral base 
$\mathscr{B}^r_X:=\mathscr{H}om(X,\mathscr{B}^r)
\times_{\mathscr{H}om(X,\mathrm{B}\gl_d)}\mathrm{Spec}(\kk)$,
where the cotangent bundle $\Omega^1_X$ corresponds to the morphism $\mathrm{Spec}(\kk)\to\mathscr{H}om(X,\mathrm{B}\gl_d)$. 

We identify the Hitchin base with the affine space
$\bigoplus_{i=1}^r H^0(X,\mathrm{Sym}^i\Omega^1_X)$ and the spectral base with the space of sections
$X\to\mathrm{Chow}^r(\Omega^1_X/X)$ of the natural map
$\mathrm{Chow}^r(\Omega^1_X/X)\to X$, where 
$\mathrm{Chow}^r(\Omega^1_X/X):=
(\Omega^1_X\times_X\dots\times_X\Omega^1_X)\sslash S_r$ 
is the relative Chow scheme that classifies 
$0$-dimensional cycles of length $r$ 
in the fibers of the natural morphism $\Omega^1_X\to X$. 
From the closed immersion (\ref{closed-embedding}),
we obtain a closed immersion
$\mathscr{B}^r_X\hookrightarrow\mathscr{A}^r_X$.

Next, we recall the construction of the spectral cover
corresponding to a point $s\in\mathscr{B}^r_X$ in the $\gl_r$ case,
as outlined in \cite[Section 6]{CN20}.
Consider the Cayley morphism
\begin{align*}
    \chi: B^r\times\mathbb{A}^d\to \mathrm{Sym}^r(\mathbb{A}^d)
\end{align*}
given by 
\[
\chi([\omega_1,...,\omega_r],y)=(y-\omega_1)\cdots(y-\omega_r)=y^r-\sigma_1(\omega)y^{r-1}+\sigma_2(\omega)y^{r-2}+\dots+(-1)^r\sigma_r(\omega).
\]
The Cayley scheme is then defined as the fiber over 
$0\in\mathrm{Sym}^r(\mathbb{A}^d)$ of the Cayley morphism, and is denoted by
\[
B^{r,\bullet}:=\mathrm{Cayley}^r(\mathbb{A}^d):=
\chi^{-1}(0)\subset B^r\times\mathbb{A}^d.
\]
There is a natural projection map 
$\pr:B^{r,\bullet}\to B^r$
called the \emph{universal spectral cover} of
$B^r$.

By \cite[Proposition 6.1, (1) and (2)]{CN20} the morphism $\pr$ is finite, surjective, and \'etale exactly
over the open subset $B^{r,\circ}$ of $B^r$
consisting of those elements $[\omega_1,\dots,\omega_r]$
without multiplicity, i.e., $\omega_i\neq\omega_j$ for $i\neq j$.
Since $\pr$ is $\gl_d$-equivariant, we obtain a morphism $[B^{r,\bullet}/\gl_d]\to[B^r/\gl_d]$, which is \'etale exactly 
over the open substack $[B^{r,\circ}/\gl_d]\subset[B^r/\gl_d]$.

Recall that for each $\kk$-point $s\in\mathscr{B}^r_X$,
we have a morphism $s:X\to[B^r/\gl_d]$.
The \emph{spectral cover} $X_s$ corresponding to $s$
is defined as the following fiber product
\[
\begin{tikzcd}
X_s\arrow{d}[swap]{\pi_s}\arrow{r} &
{[B^{r,\bullet}/\gl_d]}\arrow{d}{}\\
X\arrow{r}{s} & {[B^r/\gl_d]}.
\end{tikzcd}
\]
 
The fiber of $\pi_s$ over any point 
$x\in X$ is the set of roots of the polynomial 
$P(s):=y^r-s_1y^{r-1}+s_2y^{r-2}+\dots+(-1)^rs_r$ at $x$. 
The spectral cover factors as $\pi_s:X_s\to\Omega^1_X\to X$, where 
$\Omega^1_X\to X$ is the natural projection and
we view $\Omega^1_X$ as the scheme corresponding
to the $\mathcal{O}_X$-algebra
$\bigoplus_{i=0}^{\infty} \mathrm{Sym}^{i}(\Omega^1_X)^{\lor}$.

Note that $X_s\to X$ is finite
since $[B^{r,\bullet}/\gl_d]\to[B^r/\gl_d]$ is finite. 
The coherent sheaf $\pi_{s*}\mathcal{O}_{X_s}$ on $X$ 
is naturally equipped with a 
Higgs field whose characteristic polynomial is $P(s)$.

However, $\pi_s$ is not necessarily flat because
$X_s$ is not necessarily Cohen-Macaulay, so $\pi_{s*}\mathcal{O}_{X_s}$ may not be locally free.
More precisely, the morphism $\pi_s$ is \'etale over the open subset $X^\circ$ of $X$
that is mapped into $[B^{r,\circ}/\gl_d]$ by $s$.
It is unclear whether $\pi_s$ is flat over the locus $X\setminus X^\circ$.

Although the spectral base can be described quite explicitly
as a closed subscheme of the Hitchin base,
little is known about it's geometry and topology.
For example, it is not known whether $\mathscr{B}^r_X$ is irreducible.

We know that $\mathscr{B}^r_X$ admits a $\mathbb{G}_m$-action,
given by scaling. As a consequence, we observe that it is connected.

\begin{lemma}
    The spectral base is connected.
\end{lemma}
\begin{proof}
Choose any point $s\in\mathscr{B}^r_X(\kk)$.
We view the $\mathbb{G}_m$-orbit of $s$ as a morphism
$s:\mathbb{G}_m\to\mathscr{B}^r_X$, $t\mapsto t\cdot s$.
Since $\mathscr{B}^r_X\subset\mathscr{A}^r_X$ is closed,
the limit of this orbit as $t\to0$ exists,
and is precisely the origin $0\in\mathscr{B}^r_X$.

Thus, the closure of every $\mathbb{G}_m$-orbit contains $0$,
and we conclude that $\mathscr{B}^r_X$ is connected.
\end{proof}

\subsection{Irreducible spectral data}
In this subsection we define a stratification of the spectral base
and recall the discriminant. We also make some preliminary observations on irreducible spectral data.

Each $\kk$-point $s\in\mathscr{B}^r_X$ is an $r$-tuple $(s_1,\dots,s_r)$,
with $s_i\in H^0(X,\mathrm{Sym}^i\Omega^1_X)$,
which we may view as a homogeneous polynomial $P(s)$
of degree $r$ in a formal variable $y$
\begin{align*}
    P(s):= y^r-s_1y^{r-1}+\dots+(-1)^rs_r.
\end{align*}
We can define morphisms of Chow schemes 
\begin{align*}
B^p\times B^q&\to B^r\\
([\omega_1,\dots,\omega_p],[\psi_1,\dots,\psi_q])&\mapsto[\omega_1,\dots,\omega_p,\psi_1,\dots,\psi_q]
\end{align*}
for all $p,q\in\mathbb{N}$ such that $p+q=r$.
Note that this morphism is $\gl_d$-invariant.
This yields a morphism
\begin{align*}
    m_{p,q}:\mathscr{B}^p_X\times\mathscr{B}^q_X \to
    \mathscr{B}^r_X,
\end{align*}
which sends an $S$-point on the left-hand side, corresponding
to $\gl_d$-invariant morphisms $f:\Omega_X^1\times S \to B^p$
and $g:\Omega^1_X \times S \to B^q$, where we view $\Omega^1_X$
as a $\gl_d$-torsor over $X$, to 
$f\times g:\Omega_X^1 \times S \to B^p\times B^q$.
On the level of $\kk$-points, $m_{p,q}$ is 
given by multiplication of characteristic polynomials. Note that we can analogously define a multiplication map $\mathscr{A}^p_X\times\mathscr{A}^q_X\to\mathscr{A}^r_X$ 
on the Hitchin base,
which we use in the proof of Lemma \ref{lemma-proper}.

We define $(p,q)\subset\mathscr{B}^r_X$
as the subset consisting of those $s$ such that $P(s)$
decomposes as a product a degree $p$ polynomial in $\mathscr{B}^p_X$
and a degree $q$ polynomial in $\mathscr{B}^q_X$.
We define $(r)\subset\mathscr{B}^r_X$ as the subset consisting
of irreducible polynomials, i.e., those that do not decompose.
Thus, the spectral base can be written as the union 
\begin{align}\label{strat-spectralbase}
    \mathscr{B}^r_X=\bigcup_{p+q=r;\; p\le q}(p,q)\cup(r).
\end{align}

It is straightforward to check that the set-theoretic 
image (i.e. the image of the $\kk$-points) of $m_{(p,q)}$ is precisely the subset $(p,q)$.
Using the same method as the proof of \cite[Lemma 5.6]{PW26},
we show that the multiplication maps are proper.

\begin{lemma}
\label{lemma-proper}
    The morphisms $m_{(p,q)}$ are proper for all 
    $p,q\in\mathbb{N}$ such that $p+q=r$.
\end{lemma}
\begin{proof}
Any $\kk$-point $s=(s_1,\dots,s_p)$ of the Hitchin base
$\mathscr{A}^p_X$ can also be expressed as a homogeneous polynomial 
$y^p-s_1y^{p-1}+\dots+(-1)^ps_p$ and we can define multiplication maps
$\oplus_{(p,q)}:\mathscr{A}^p_X\times \mathscr{A}^q_X\to \mathscr{A}^r_X$.
Since $\mathscr{B}^p_X$ is a closed subscheme of $\mathscr{A}^p_X$, 
we have closed immersions
$\iota_{(p,q)}:\mathscr{B}^p_X\times\mathscr{B}^q_X
\to\mathscr{A}^p_X\times\mathscr{A}^q_X$,
and it is easy to check that $m_{(p,q)}=\oplus_{(p,q)}\circ\iota_{(p,q)}$.

We show that the map $\oplus_{(p,q)}$ is proper,
from which it immediately follows that $m_{(p,q)}$ is proper.

Let $\mathscr{P}^r_X$ denote the projective completion of $\mathscr{A}^r_X$
formed by adding an additional variable in degree zero.
More precisely, $\mathscr{P}^r_X:=
(\bigoplus_{i=0}^rH^0(X,\mathrm{Sym}^i\Omega^1_X)\setminus
\{\mathbf{0}\})\sslash \mathbb{G}_m$,
where $H^0(X,\mathrm{Sym}^0\Omega^1_X)=H^0(X,\mathcal{O}_X)$,
$\mathbf{0}:=
(0,\dots,0)\in\bigoplus_{i=0}^rH^0(X,\mathrm{Sym}^i\Omega^1_X)$ 
is the origin, and the $\mathbb{G}_m$-action is given by scaling, i.e.,
$\lambda\cdot(s_0,\dots,s_r)=(\lambda s_0,\dots,\lambda s_r)$,
$\lambda\in\mathbb{G}_m$. 
Thus, a point $s\in\mathscr{P}_X^r$ can be represented by 
$[s_0:\dots:s_r]$ in homogeneous coordinates
and $\mathscr{A}_X^r$ can be identified with the chart  $\{s_0\neq 0\}$. 

We define a map 
$\overline{\oplus}_{(p,q)}:
\mathscr{P}^p_X\times \mathscr{P}^q_X\to\mathscr{P}^r_X$ 
extending $\oplus_{(p,q)}$, given by
\[
    ([s_0:s_1:\dots:s_p],[t_0:t_1:\dots:t_q])\mapsto[v_0:v_1:\dots:v_r],
\]
where each $v_i$ is $(-1)^i$ times the degree $r-i$ coefficient of the polynomial 
\[
    (s_0y^p-s_1y^{p-1}+\dots+(-1)^ps_p)\cdot
    (t_0y^q-t_1y^{q-1}+\dots+(-1)^qt_q),
\]
for $0\le i\le p+q=r$. 

Note that $\overline{\oplus}_{(p,q)}$ is indeed well defined,
because the product of two such polynomials is identically zero
if and only if one of them is identically zero.
We recover $\oplus_{(p,q)}$ by restricting $\overline{\oplus}_{(p,q)}$
to the chart $\{s_0\neq0\}\times\{t_0\neq0\}$.

We have the following Cartesian diagram

\begin{center}
\begin{tikzcd}
       \mathscr{A}^p_X\times\mathscr{A}^q_X \ar[hookrightarrow]{r} \ar{d}{\oplus} & \mathscr{P}^p_X\times\mathscr{P}^q_X \ar{d}{\overline{\oplus}}\\
        \mathscr{A}^r_X \ar[hookrightarrow]{r} & \mathscr{P}^r_X
\end{tikzcd}
\end{center}
from which it follows that $\oplus_{(p,q)}:\mathscr{A}^p_X\times
\mathscr{A}^q_X\to\mathscr{A}^r_X$ is proper.
\end{proof}

In particular, the subsets $(p,q)$ are closed,
while $(r)$ is open in $\mathscr{B}^r_X$. 

\begin{definition}[Irreducible spectral datum]
    We call a $\kk$-point of $(r)$ an 
    \emph{irreducible spectral datum}.
\end{definition}

Let $0\neq s\in\mathscr{B}^r_X$ be a $\kk$-point.
Recall that the roots of the associated characteristic polynomial
$P(s)$ at any point $x\in X$ are precisely the elements 
$\{\omega_1,\dots,\omega_r\}$ of the
fiber $\pi_s^{-1}(x)\subset X_s$ of the spectral cover 
$\pi_s:X_s\to X$, counted with multiplicity. 

We observe that the irreducibility of $P(s)$ is closely related to the irreducibility of the spectral cover $X_s$.
The following argument is closely related to and 
inspired by \cite[Section 2]{HLM24}.

\begin{lemma}
    \label{lemma-generically-integral}
    Let $s\in\mathscr{B}^r_X$ be a closed point. 
    If $s\in(r)$, then the spectral cover 
    $X_s$ has exactly one irreducible
    component dominating $X$ which
    is in addition generically integral. 
    Furthermore, the morphism $\pi_s:X_s\to X$
    is generically \'etale.
\end{lemma}
\begin{proof}
    As we are in characteristic $0$ a finite surjective morphism
    of integral schemes is generically \'etale
    and it suffices to show the first part of the lemma.
    
    Consider the irreducible components $X_{s,1},\dots,X_{s,n}$
    of the spectral cover $X_s\to X$
    of dimension $d=\dim X$.
    Let $W\subset X$ be the union of images of pairwise intersections
    $X_{s,i}\cap X_{s,j}$ of irreducible components for $1\le i<j\le n$.
    Set $U:= X\setminus W$, $U_s:=\pi_s^{-1}(U)$, 
    and note that $U$ is a dense open subset of $X$.

    Let $U_s=\bigcup_{k=1}^n U_{s,k}$ be a decomposition 
    into irreducible components,
    which is actually a disjoint union because
    $U_{s,i}\cap U_{s,j}=\emptyset$ for all $1\le i<j\le n$.
    Consider the coherent Higgs sheaf
    $\pi_{s*}\mathcal{O}_{X_s}\to\pi_{s*}\mathcal{O}_{X_s}\otimes\Omega^1_X$
    whose characteristic polynomial is $P(s)$.
    Restricted to $U$, we obtain a direct sum decomposition 
    
    \[
    (\pi_{s*}\mathcal{O}_{X_s})_{\mid U}
    \cong\bigoplus_{k=1}^n\mathcal{F}_{U,k}
    \]
    as Higgs sheaves on $U$, where $\mathcal{F}_{U,k}$
    is the direct image of $\mathcal{O}_{U_{s,k}}$ along $U_{s,k}\to U$.
    We obtain a factorization of $P(s)$ over $U$ as
    $P(s)_{\mid U}=\prod_{k=1}^nP_{U,k}$,
    where $P_{U,k}$ denotes the characteristic polynomial
    of $\mathcal{F}_{U,k}$. 
    
    Let $X_{s,k}$ denote the scheme-theoretic closure of $U_{s,k}$
    in $X_s$, and consider the natural morphisms $\pi_{s,k}:X_{s,k}\to X$.
    Via pushforward we obtain coherent Higgs sheaves
    $\mathcal{F}_k:=\pi_{s,k*}\mathcal{O}_{X_{s,k}}$ on $X$ which
    extend the Higgs sheaves $\mathcal{F}_{U,k}$, that is, 
    $\mathcal{F}_k|_U=\mathcal{F}_{U,k}$ for $1\le k\le n$.
    Let $t_k$ be the image via the Hitchin morphism of 
    $\mathcal{F}_k$ and note that $t_k\in\mathscr{B}^{l_k}_X$,
    where $l_k\in\mathbb{N}$ is the rank of $\mathcal{F}_k$.
    Indeed, locally the roots of the characteristic polynomial
    of each $\mathcal{F}_k$ are sections of $\Omega^1_X$.
    Associating to each point $x\in X$ the unordered tuple of these 
    roots defines a map $t_k:X\to\mathrm{Chow}^{l_k}(\Omega^1_X/X)$.

    Consider the characteristic polynomial $\prod_{k=1}^n P(t_k)$
    of $\bigoplus_{k=1}^n\mathcal{F}_k$.
    Note that we have $P(t_k)|_U=P_{U,k}$ for each $k$ and thus 
    \[
    \prod_{k=1}^n P(t_k)|_U=\prod_{k=1}^n P_{U,k}. 
    \]
    It follows that $P(s)$ and $\prod_{k=1}^n P(t_k)$ agree over $U$.
    Since $U\subset X$ is dense open,
    it follows that $P(s)=\prod_{k=1}^n P(t_k)$ holds on $X$.
    Since we assumed that $s$ is irreducible, we conclude $n=1$, that is,
    the spectral cover $X_s$ has exactly one irreducible component that dominates $X$.
    
    Let $\widehat{X}_s$ denote the reduced subscheme underlying $X_s$
    and consider the natural morphisms $\iota:\widehat{X}_s\to X_s$
    and $\widehat{\pi}_s:=\pi_s\circ\iota:\widehat{X}_s\to X$.
    There is a commutative diagram
    \begin{center}
    \begin{tikzcd}
        \mathcal{O}_{X_s}\ar{r}{\lambda}\ar{d}{} &
        \mathcal{O}_{X_s}\otimes\pi_s^*\Omega^1_{X_s}\ar{d}{}\\
        \iota_*\iota^*\mathcal{O}_{X_s}\ar{r}{\iota_*\iota^*\lambda} &
        \iota_*\iota^*(\mathcal{O}_{X_s}\otimes\pi_s^*\Omega^1_{X_s})
    \end{tikzcd}
    \end{center}
    of sheaves on $X_s$, where $\lambda$ denotes 
    the tautological section of $\pi_s^*\Omega^1_X$.
    Pushing forward by $\pi_s$,
    we obtain a surjective morphism 
    $\pi_{s,*}\mathcal{O}_{X_s}\to
    \widehat{\pi}_{s,*}\mathcal{O}_{\widehat{X}_s}$
    of Higgs sheaves on $X$.
    Note that $X_s$ is generically reduced if and only if
    $\pi_{s,*}\mathcal{O}_{X_s}$ and $\widehat{\pi}_{s,*}\mathcal{O}_{\widehat{X}_s}$ have the same
    rank. Denote by $\mathcal{K}$ the kernel 
    of $\pi_{s,*}\mathcal{O}_{X_s}\to
    \widehat{\pi}_{s,*}\mathcal{O}_{\widehat{X}_s}$
    and note that $\mathcal{K}$ is a Higgs subsheaf of
    $\pi_{s,*}\mathcal{O}_{X_s}$. 
    
    Let $P(\mathcal{K})$ (resp. $P(\widehat{\pi}_{s,*}(\mathcal{O}_{\widehat{X}_s}))$) denote the characteristic polynomials of $\mathcal{K}$ (resp. $\widehat{\pi}_{s,*}(\mathcal{O}_{\widehat{X}_s})$).
    If $\mathcal{K}$ had positive rank, we would obtain a factorization of
    characteristic polynomials 
    $P(s)=P(\mathcal{K})P(\widehat{\pi}_{s,*}(\mathcal{O}_{\widehat{X}_s}))$
    since the characteristic polynomial is multiplicative 
    in short exact sequences.
    Since $s$ is irreducible, we obtain a contradiction,
    that is, we obtain that $\mathcal{K}$ has rank $0$.
    Equivalently, $\pi_{s,*}\mathcal{O}_{X_s}$
    and $\widehat{\pi}_{s,*}\mathcal{O}_{\widehat{X}_s}$
    have the same rank and we conclude.
\end{proof}

For any $\kk$-point $\omega=[\omega_1,\dots,\omega_r]\in B^r$, we can define its \emph{discriminant} as the expression $D(\omega):=\prod_{i<j}(\omega_i-\omega_j)^2\in\mathrm{Sym}^{r(r-1)}\mathbb{A}^d$. Note that this is the same as the discriminant of the polynomial $\prod_{i=1}^r(y-\omega_i)$, where $y$ is a formal variable, hence the name. By \cite[Proposition 6.1]{CN20}, the morphism $B^{\bullet,r}\to B^r$ is \'etale precisely over the open subset $B\setminus\{\omega\;|\;D(\omega)=0\}$.

The assignment $\omega\mapsto D(\omega)$ defines a $\gl_d$-invariant morphism $B^r\to\mathrm{Sym}^{r(r-1)}{\mathbb{A}^d}$. Thus, we obtain a morphism of quotient stacks 
\[
\mathscr{B}^r=[B^r/\gl_d]\to[\mathrm{Sym}^{r(r-1)}{\mathbb{A}^d}/\gl_d]=:\mathscr{D}^r.
\]
This further induces a morphism of Hom-stacks
\[
D:\mathscr{H}om(X,\mathscr{B}^r)
\times_{\mathscr{H}om(X,\mathrm{B}\gl_d)}\mathrm{Spec}(\kk)\to \mathscr{H}om(X,\mathscr{D}^r)
\times_{\mathscr{H}om(X,\mathrm{B}\gl_d)}\mathrm{Spec}(\kk),
\]
where the morphism $\mathrm{Spec}(\kk)\to\mathscr{H}om(X,\mathrm{B}\gl_d)$ again corresponds to the cotangent bundle $\Omega^1_X$. The target can be identified with the affine space $H^0(X,\mathrm{Sym}^{r(r-1)}\Omega^1_X)$, so we have a morphism 
\[
D:\mathscr{B}^r_X\to H^0(X,\mathrm{Sym}^{r(r-1)}\Omega^1_X)
\] 
that is given on $\kk$-points by sending $s=(s_1,\dots,s_r)$ to the discriminant $D(s)$ of the polynomial $P(s)$. We call $D$ the \emph{discriminant map}.

Note that $D(s)$ is a homogeneous expression in the $s_i$. For example, when $r=2$, we have $s=(s_1,s_2)\in\mathscr{B}^2_X$
and $D(s)=s_1^2-4s_2\in H^0(X,\mathrm{Sym}^2\Omega^1_X)$. By definition, $D(s)$ vanishes at a point $x$
if and only if $P(s)$ has roots with multiplicity $>1$ at $x$,
i.e., if and only if $\omega_i(x)=\omega_j(x)$ for some $i\neq j$.

We observe that for any nonzero $s\in\mathscr{B}^r_X$,
the polynomial $P(s)$ can be decomposed 
as a product of irreducible polynomials,
such that each irreducible factor has nonzero discriminant.

\begin{lemma}\label{zerodisc}
    Let $0\neq s\in\mathscr{B}^r_X$ be a $\kk$-point.
    Then $s$ can be written as a product $s=\prod_{k=1}^n t_k^{m_k}$ 
    of irreducible factors
    $t_k\in\mathscr{B}^{l_k}_X(\kk)$, $m_k\in\mathbb{N}$,
    each of which has nonzero discriminant.
\end{lemma}
\begin{proof}
    Let $s=\prod_{k=1}^n t_k^{m_k}$
    be a decomposition of $s$ into irreducible factors.
    Then $t_k\in(l_k)\subset\mathscr{B}^{l_k}_X$
    for some $l_k\in\mathbb{N}$ and $m_k\in\mathbb{N}$
    denotes the multiplicity of $t_k$ in $s$ for $1\le k\le n$.

    We know from Lemma \ref{lemma-generically-integral} 
    that the spectral cover $\pi_{t_k}:X_{t_k}\to X$
    corresponding to $t_k$ is generically \'etale of degree $l_k$.
    Thus, there is a dense open subset $X^\circ_k\subset X$
    over which $\pi_{t_k}$ is \'etale. 
    
    Recall that the fibers of $\pi_{t_k}$ consist of roots of $P(t_k)$,
    which are all pairwise distinct over $X^\circ_k$.
    This means that $D(t_k)$ is nonzero over $X^\circ_k$
    and is thus a nonzero section of 
    $\mathrm{Sym}^{l_k(l_k-1)}\Omega^1_X$ for $1\le k\le n$.
\end{proof}

\begin{remark}\label{red-step}
    We can draw the following conclusions from Lemma \ref{zerodisc}.
    \begin{itemize}
    \item If $s\in\mathscr{B}^r_X$ is an irreducible 
    $\kk$-point (i.e. if $s\in(r)$),
    then $D(s)\neq 0$.
    \item In order to prove Conjecture \ref{conjecture},
    it suffices to prove that for every point
    $s\in(r)\subset\mathscr{B}^r_X$, the fiber $(h^r_X)^{-1}(s)$
    is non-empty.
    \end{itemize}
\end{remark}

\section{Normalized spectral covers}
\label{sec-normalized-spectral-covers}
In this section we modify the construction of
a spectral cover of an irreducible spectral datum
to obtain a normal spectral cover.
We also observe that the Hitchin morphism
from the stack of reflexive Higgs sheaves
factors via the spectral base and that this morphism is onto.

Let $X$ be a smooth proper variety 
and let $r\in\mathbb{N}$ be fixed.
Let $s\in\mathscr{B}^r_X$ be a closed point.
In view of Remark \ref{red-step}, we assume that $s\in(r)$.
Henceforth, we slightly abuse notation
by using $s$ and $P(s)$ interchangeably. 

Let $\widetilde{X}_s$ denote the reduced subscheme underlying the unique irreducible component of $X_s$ that is mapped onto $X$ via $\pi_s$. Note that $\widetilde{X}_s$ is isomorphic to $X_s$ over the dense open subset $X^\circ\subset X$ where $\pi_s$ is \'etale, and therefore $\widetilde{X}_s$ is integral and generically \'etale over $X$ by Lemma \ref{lemma-generically-integral}. 

Recall from the last paragraph of Section \ref{sec-spectral covers} that, in general, $\widetilde{X}_s$ may not be normal or Cohen-Macaulay.
Let $\nu:X'_s\to \widetilde{X}_s$ denote the normalization map.
We call $X'_s$ the \emph{normalized spectral cover} corresponding to $s$. 

We have the following commutative triangle
\[\begin{tikzcd}
    X'_s\arrow{r}{\pi'_s}\arrow{d}[swap]{\nu} & X\\
    \widetilde{X}_s\arrow{ur}[swap]{\pi_s}
\end{tikzcd}\]
where $\pi'_s=\pi_s\circ\nu$ is finite and surjective. 

Let $Z\subset X$ and $Z'\subset X$
denote the loci where the morphisms $\pi_s$
and $\pi'_s$ are branched, respectively, and note that $Z=X\setminus X^\circ$.
It follows from the construction of the spectral cover
that the morphism $\pi_s$ is ramified precisely over those points $x\in X$
where $P(s)$ has roots with multiplicity $>1$ at $x$ (\cite[Proposition 6.1]{CN20}).
In other words, we have the following description.
\begin{align*}
    Z=\{x\in X\;|\;D(s)(x)=0\}\subset X.
\end{align*}

 Next, we observe that passing from $\widetilde{X}_s$
 to its normalization $X'_s$ does not add any additional branching.

\begin{lemma}\label{branchlocus}
    With notation as above, we have $Z'\subset Z$.
\end{lemma}
\begin{proof}
Let $U$ and $U'$ denote the loci in $X$ where the maps $\pi_s:\widetilde{X}_s\to X$
and $\pi'_s:X'_s\to X$ respectively, are \'etale.
Then it is clear that $U\subset U'$.
Since $Z$ and $Z'$ are complements of $U$ and $U'$ respectively,
the claim follows.
\end{proof}

Since $X'_s$ is normal and $X$ is smooth,
it follows from the purity of branch locus
(\cite[\href{https://stacks.math.columbia.edu/tag/0BMB}{Tag 0BMB}]{sp})
that either $Z'$ has pure codimension one in $X$, or $Z'$ is empty.  

We want to determine sufficient conditions such that,
for a general $\kk$-point $s$ in $(r)\subset\mathscr{B}^r_X$,
the normalized spectral cover $\pi'_s:X'_s\to X$ is \'etale.
We begin with the following construction.

Since $\widetilde{X}_s$ is a closed subscheme of $\Omega^1_X$,
the vector bundle $\pi_s^*\Omega^1_X$ on $\widetilde{X}_s$
has a tautological global section $\lambda\in H^0(\widetilde{X}_s,\pi_s^*\Omega^1_X)$.
Multiplying by $\lambda$ gives a natural 
$\pi_s^*\Omega^1_X$-twisted bundle on $\widetilde{X}_s$
\begin{align*}
    \mathcal{O}_{\widetilde{X}_s}\xrightarrow{\times\lambda}
    \mathcal{O}_{\widetilde{X}_s}\otimes\pi_s^*\Omega^1_X.
\end{align*}
Pushing forward along $\pi_s$, we obtain a Higgs sheaf on $X$ by \cite[Lemma 6.8]{simp2}
\begin{align*}
    \theta:\pi_{s*}\mathcal{O}_{\widetilde{X}_s}
    \to\pi_{s*}(\mathcal{O}_{\widetilde{X}_s}\otimes\pi_s^*\Omega^1_X)\cong
    \pi_{s*}\mathcal{O}_{\widetilde{X}_s}\otimes\Omega^1_X.
\end{align*}

By the generalization of the BNR correspondence (\cite[Proposition 3.6]{BNR}) to higher dimensional varieties (\cite[Proposition 6.3]{CN20}), the characteristic polynomial of this Higgs sheaf is precisely $s$.

We now look at the vector bundle 
$\nu^*\pi_s^*\Omega^1_X={\pi'_s}^*\Omega^1_X$ on $X'_s$,
and the section $\nu^*\lambda\in H^0(X'_s,{\pi'_s}^*\Omega^1_X)$.
As before, multiplying by $\nu^*\lambda$ gives a natural
${\pi'_s}^*\Omega^1_X$-twisted bundle on $X'_s$ 
\begin{align*}
    \mathcal{O}_{X'_s}\xrightarrow{\times\nu^*\lambda}
    \mathcal{O}_{X'_s}\otimes{\pi'_s}^*\Omega^1_X.
\end{align*}
Pushing forward along $\pi'_s$,
we obtain an $\Omega^1_X$-twisted sheaf on $X$
\begin{align*}
    \theta':\pi'_{s*}\mathcal{O}_{X'_s}\to\pi'_{s*}(\mathcal{O}_{X'_s}\otimes{\pi'_s}^*\Omega^1_X)\cong\pi'_{s*}\mathcal{O}_{X'_s}\otimes\Omega^1_X.
\end{align*}

\begin{lemma}\label{higgssheaf}
    The pair $(\pi'_{s*}\mathcal{O}_{X'_s},\theta')$
    is a Higgs sheaf on $X$ with spectral datum $s$.
\end{lemma}
\begin{proof}
    Recall that we have assumed $s\in(r)$ and that $\widetilde{X}_s$ is reduced.
    The morphisms $\pi_s$ and $\pi'_s$ are both \'etale over 
    the dense open subset $U=X\setminus Z$ of $X$,
    and clearly $\pi_s^{-1}(U)\cong{\pi_s'}^{-1}(U)$.

    In particular, we have
    $(\pi'_{s*}\mathcal{O}_{X'_s}|_U,\theta'|_U)\cong
    (\pi_{s*}\mathcal{O}_{\widetilde{X}_s}|_U,\theta|_U)$
    as $\Omega^1_X|_U$-twisted bundles.
    Since $(\pi_{s*}\mathcal{O}_{\widetilde{X}_s}|_U,\theta|_U)$ is a Higgs bundle over $U$,
    we have the vanishing
    $(\theta\wedge\theta)|_U=\theta|_U\wedge\theta|_U=0$,
    from which it follows that $(\theta'\wedge\theta')|_U=0$.
    Since $\theta'\wedge\theta'$ vanishes over the dense open subset $U$, we conclude that $\theta'\wedge\theta'=0$ on all of $X$.

    Let $s'=(s'_1,...,s'_r)$ be the spectral datum associated 
    with $(\pi'_{s*}\mathcal{O}_{X'_s},\theta')$. Since 
    $(\pi_{s*}\mathcal{O}_{\widetilde{X}_s}|_U,\theta|_U)\cong
    (\pi'_{s*}\mathcal{O}_{X'_s}|_U,\theta'|_U)$
    as Higgs bundles on $U$, the characteristic polynomials of $\theta|_U$ and $\theta'|_U$
    agree. Therefore, we have
    $s_i|_U=s'_i|_U$ for all $1\le i\le r$. 

    Again, $U\subset X$ being a dense open subset,
    $s_i|_U=s'_i|_U$ implies that $s_i=s'_i$ 
    for all $1\le i\le r$, i.e., $s'=s$.
\end{proof}

\begin{remark}\label{reflexive}
    Since $X'_s$ and $X$ are both normal and
    $\pi'_s$ is a finite surjective morphism, 
    $\pi'_{s*}\mathcal{O}_{X'_s}$ is in fact a
    \emph{reflexive} Higgs sheaf on $X$  
    (\cite[Corollary 1.7]{reflexive-hartshorne}). 
\end{remark}

This recovers a result of He, Liu, and Mok 
(see \cite[Proposition 3.5]{HLM24}),
which says that for every $\kk$-point $s\in\mathscr{B}^r_X$,
there is a reflexive Higgs sheaf whose associated 
spectral datum is precisely $s$.

Let $\mathscr{M}^{r,\mathrm{refl}}_X$ denote the stack 
of reflexive Higgs sheaves of rank $r$ over $X$ which is an
open substack of the stack $\mathscr{M}^{r,\mathrm{tf}}_X$ 
of torsion-free Higgs sheaves of rank $r$ over $X$.
We denote the Hitchin morphisms by
$h^{r,\mathrm{tf}}_X:\mathscr{M}^{r,\mathrm{tf}}_X\to\mathscr{A}^r_X$ and $h^{r,\mathrm{refl}}_X:\mathscr{M}^{r,\mathrm{refl}}_X\to\mathscr{A}^r_X$.

\begin{lemma}\label{factorization}
    The morphism $h^{r,\mathrm{tf}}_X$
    factors through the spectral base $\mathscr{B}^r_X$.
\end{lemma}
\begin{proof}
    Let $S$ be a $\kk$-scheme of finite type
    and consider an $S$-point
    $(\mathcal{E},\varphi)\in\mathscr{M}^{r,\mathrm{tf}}_X(S)$,
    that is, $(\mathcal{E},\varphi)$ is a family of reflexive Higgs
    sheaves on $X\times S$ flat over $S$.

    Fix a closed point $s\in S$ and consider the 
    torsion-free sheaf $\mathcal{E}_s$ as the base change
    of $\mathcal{E}$ along $X\xrightarrow{\id \times s}X\times S$.
    Then $\mathcal{E}_s$ is a vector bundle over 
    a big open subset $U(s)\subset X$.
    Since being a family of vector bundles is an open condition
    we find an open neighborhood $V(s)\subset S$ of $s$
    such that $\mathcal{E}_{\mid U(s)\times V(s)}$ is a family
    of vector bundles flat over $V(s)$.

    Since $S$ is quasi-compact, there exist
    finitely many closed points $s_1,\dots, s_n$ such that $V(s_1),\dots,V(s_n)$
    cover $S$. In particular, setting $U:=\bigcap_{i=1}^n U(s_i)$
    yields a big open subset of $X$ such
    that $\mathcal{E}_{\mid U\times S}$ is a family of vector bundles
    on $U$ flat over $S$.

    Then $h_{U}^r(\mathcal{E}_{\mid U\times S})$ 
    lies in $\mathscr{B}^r_U(S)$.
    Since $\mathscr{B}^r_X\to \mathscr{B}^r_U$
    is an isomorphism by \cite[Lemma 5.2]{SS24}
    we conclude.
\end{proof}

Combining Remark \ref{reflexive} and Lemmas \ref{higgssheaf} 
and \ref{factorization},
we arrive at the following weaker version
of the statement of Conjecture \ref{conjecture}

\begin{proposition}\label{weakprop}
    The morphism $h^{r,\mathrm{refl}}_X$ is surjective 
    onto the spectral base $\mathscr{B}^r_X$.
\end{proposition}

\begin{remark}
    This allows us to formulate Conjecture \ref{conjecture}
    without using the spectral base as:
    the set-theoretic images of $h^r_X$ and $h^{r,\mathrm{refl}}_X$ agree. 
\end{remark}

In particular, restricting to the case $\mathrm{dim}(X)=2$,
and observing that a reflexive sheaf on a smooth scheme
is locally free outside a subset of codimension three,
we see that Chen-Ng\^o's conjecture is true for surfaces 
(see also \cite[Theorem 1.2]{SS24}).
Alternatively, one can also use the fact that 
a normal surface is Cohen-Macaulay
(\cite[\href{https://stacks.math.columbia.edu/tag/033P}{Tag 033P}]{sp}),
and hence by miracle flatness
(\cite[\href{https://stacks.math.columbia.edu/tag/00R4}{Tag 00R4}]{sp})
the map $\pi'_s$ is flat.
This recovers the main result of Song and Sun.

\begin{cor}[{\cite[Theorem 1.2]{SS24}} and {\cite[Remark 2.16]{HLM24}}]
Let $X$ be a smooth, proper surface.
Then the spectral base $\mathscr{B}^r_X$ 
is equal to the set-theoretic image of $h^r_X$ for all $r$.
\end{cor}

\subsection{Birational morphisms}

Let $f:Y\to X$ be a birational morphism between smooth proper schemes.
The natural morphism $f^*\Omega^1_X\to\Omega^1_Y$ induces morphisms
$\mathrm{Sym}^kf^*\Omega^1_X\to\mathrm{Sym}^k\Omega^1_Y$ for all $k\ge1$.
Taking global sections and composing with
$H^0(X,\mathrm{Sym}^k\Omega^1_X)\to H^0(X,\mathrm{Sym}^kf^*\Omega^1_X)$
for $1\le k\le r$,
gives a map $\mathscr{A}^r_X\to\mathscr{A}^r_Y$ between Hitchin bases, which we denote by $f^*$. Restricting to $\mathscr{B}^r_X$, we obtain a morphism $f^*:\mathscr{B}^r_X\to\mathscr{B}^r_Y$ between spectral bases, which is actually an isomorphism by \cite[Lemma 2.6]{HLM24} or \cite[Theorem 1.4]{SS24}.

\begin{lemma}
    Let $X$ be a smooth proper variety of dimension $d$
    and fix $r\in\mathbb{N}$.
    Consider a $\kk$-point $s\in \mathscr{B}^r_X$.
    Then there exists a birational morphism
    $\pi:Y\to X$ of smooth proper varieties
    such that $\pi^{*}s$ lies in the image of
    $h^r_Y$.
\end{lemma}

\begin{proof}
    Via decomposing $s$ into irreducible factors
    and applying \Cref{weakprop}
    we obtain a reflexive Higgs sheaf
    $\mathcal{F}$ on $X$ with characteristic polynomial $s$. 
    By \cite[\href{https://stacks.math.columbia.edu/tag/0ESN}{Tag 0ESN}]{sp}
    we can find a birational morphism $\pi':Y'\to X$ such that
    the strict transform of $\pi'^{\ast}\mathcal{F}$ is locally free. 

    Since we are in characteristic $0$ we can further
    resolve the singularities of $Y'$
    to find a birational morphism $\pi:Y\to X$
    with smooth proper $Y$ such that 
    the strict transform $\pi^{[\ast]}\mathcal{F}$
    of $\pi^{\ast}\mathcal{F}$ is locally free.

    We also can pullback the Higgs field 
    $\mathcal{F}\to \mathcal{F}\otimes \Omega^1_X$ to $Y$.
    Taking strict transform
    and composing with the natural morphism
    $\pi^{\ast}\Omega^1_X\to \Omega^1_Y$ yields
    a Higgs field on the strict transform
    $\pi^{[\ast]}\mathcal{F}$ on $Y$.

    Via restricting to the open locus where $\pi$ is an
    isomorphism we conclude that the characteristic
    polynomial of $\pi^{[\ast]}\mathcal{F}$
    coincides with $\pi^{\ast}s$.
\end{proof}

In other words, Conjecture \ref{conjecture} is equivalent to the conjecture that the set-theoretic image of the Hitchin morphism is a birational invariant. 

\section{A stronger version of the conjecture}

In this section we prove our main theorem:
the Hitchin morphism from the Dolbeault moduli space surjects onto the spectral base for $K$-trivial varieties, i.e., a stronger version of the Chen-Ng\^o conjecture holds in this case. Our approach is to study the vanishing locus of the symmetric
differential arising from the discriminant.

Let $X$ be a smooth proper variety.
Let $r\in\mathbb{N}$ be fixed and let $s\in\mathscr{B}^r_X$ 
be a $\kk$-point.
By Lemma \ref{branchlocus},
we have the inclusion $Z'\subset Z$ of branch loci of
$\pi'_s:X'_x\to X$ and $\pi_s:X_s\to X$ (normalized) spectral cover
constructed in \Cref{sec-normalized-spectral-covers},
and $Z$ is given by the vanishing locus of the discriminant $D(s)$.
If $\mathrm{codim}_X Z=1$,
then in general the (normalized) spectral cover may not be flat.
However, if $\mathrm{codim}_XZ\ge2$,
then it follows that $Z'=\emptyset$
and $\pi_s':X'_s\to X$ is \'etale by purity of branch locus.
We want to determine sufficient conditions for which the latter case occurs.

Consider the vector space $H^0(X,\mathrm{Sym}^k\Omega^1_X)$
for each $k\ge1$, and define the subsets $W^i$ as follows.

\begin{align*}
    W^i:= \{\omega\in
    H^0(X,\mathrm{Sym}^k\Omega^1_X)\;|\;\mathrm{codim}_XV(\omega)\le i\},
\end{align*}
where $V(\omega)\subset X$ denotes the vanishing locus of $\omega$. 

The $W^i$ are closed algebraic subsets of
$H^0(X,\mathrm{Sym}^k\Omega^1_X)$,
but they may not be linear unless $i=1$ and $k=1$ 
(\cite[Theorem 1.11]{DHL24}).
By definition, we have the inclusion $W^i\subset W^{i+1}$ for all $i$.
This defines a stratification of 
$H^0(X,\mathrm{Sym}^k\Omega^1_X)$ given by  
\begin{align*}
    W^1\subset W^2\subset\dots\subset W^d\subset
    H^0(X,\mathrm{Sym}^k\Omega^1_X).
\end{align*}

Consider a spectral datum $s\in\mathscr{B}^r_X(\kk)$ 
such that its image 
$D(s)$ via the discriminant map lies in $H^0(X,\mathrm{Sym}^{r(r-1)}\Omega^1_X)\setminus W^1$.
Then, by the purity of branch locus,
the normalized spectral cover $X'_s\to X$ 
is \'etale and $X'_s$ is smooth. 

For a general point $s\in\mathscr{B}^r_X$, 
it seems a difficult but interesting problem
to determine conditions such that 
$D(s)\in H^0(X,\mathrm{Sym}^{r(r-1)}\Omega^1_X)\setminus W^1$. 
This is related to the vanishing loci of symmetric differentials, 
which (to our knowledge) have not been widely studied.

\begin{lemma}\label{dense}
    Let $X$ be a smooth proper variety such that $\mathscr{B}^r_X$
    is irreducible. 
    Suppose there is a point $s\in\mathscr{B}^r_X$
    such that $D(s)\in H^0(X,\mathrm{Sym}^{r(r-1)}\Omega^1_X)\setminus W^1$.
    Then the preimage 
    $D^{-1}(\mathrm{im}(D)\setminus W^1)\subset\mathscr{B}^r_X$
    is a dense open subset.
\end{lemma}
\begin{proof}
  The hypothesis implies that the image of $D$
  is not contained entirely in $W^1$,
  therefore $\mathrm{im}(D)\cap W^1$ is a closed subset of the image.
  The complement of its preimage is
  a non-empty open, thus dense in $\mathscr{B}^r_X$ by irreducibility.
\end{proof}

\begin{remark}
    We can weaken the irreducibility condition in Lemma \ref{dense}
    to the condition that each irreducible component of 
    $\mathscr{B}^r_X$ contains a point $s$ such that 
    $D(s)\in H^0(X,\mathrm{Sym}^{r(r-1)}\Omega^1_X)\setminus W^1$
    to arrive at the same conclusion. 
\end{remark}

A natural follow-up question is:

\begin{question}
(When) is $\mathscr{B}^r_X$ irreducible?
\end{question}

There are several examples of varieties $X$ such that the spectral base 
$\mathscr{B}^r_X$ is irreducible for all $r\in\mathbb{N}$.
Any variety that has a "sufficiently negative" cotangent bundle has
$\mathscr{A}^r_X=\{0\}$ and thus $\mathscr{B}^r_X=\{0\}$.
Interestingly, quotients of bounded symmetric domains
of sufficiently large dimension have $\mathscr{B}^r_X=\{0\}$ for all $r$,
but the dimension of $\mathscr{A}^r_X$ is large for $r$ 
sufficiently large (see \cite[Theorem 1.3]{HLM24}).

It follows from \cite[Proposition 8.1]{CN20}
that ruled surfaces and non-isotrivial elliptic surfaces 
with reduced fibers satisfy $\mathscr{B}^r_X=\mathscr{A}^r_X$
for all $r\in\mathbb{N}$,
and thus have irreducible spectral base. 

The spectral base for an Abelian variety $X$ is given by
$\mathscr{B}^r_X\cong\mathrm{Chow}^r(H^0(X,\Omega^1_X))$ 
(see \cite[Example 5.1]{CN20}), and is thus irreducible. It follows from \cite[Theorem 1.4]{PW26} that hyperelliptic varieties also have irreducible spectral base.  

A nice consequence of the normalized spectral cover associated to $s$
being \'etale is that the vector bundle $\pi'_{s*}\mathcal{O}_{X'_s}$
corresponds to a representation of the \'etale fundamental group
$\pi_1^{\text{\'et}}(X)$.
More precisely, the following is true.

\begin{lemma}\label{polystability}
    Let $f:Y\to X$ be a finite \'etale cover 
    between normal proper varieties. 
    Then $f_*\mathcal{O}_Y$ is an \'etale trivializable 
    vector bundle on $X$, that is, 
    there exists a finite \'etale cover $X'\to X$
    such that the pullback of $f_*\mathcal{O}_Y$ is isomorphic to 
    the trivial bundle on $X'$. 
    
    In particular, if $X$ and $Y$ are projective, then $f_*\mathcal{O}_Y$ is a polystable vector bundle on $X$ and all its Chern classes vanish. 
\end{lemma}
\begin{proof}
    Let $g:Z\to Y\to X$ be the Galois closure of $f$.

    Applying Chow's lemma and normalization allow 
    us to find a normal projective
    variety $\tilde{X}$ together with a birational morphism
    $\pi:\tilde{X}\to X$.
    Set $\tilde{Y}:=\tilde{X}\times_X Y$ and 
    $\tilde{Z}:=\tilde{X}\times_X Z$.
    Then $\tilde{g}:\tilde{Z}\to \tilde{Y}\xrightarrow{\tilde{f}} \tilde{X}$ are 
    \'etale covers of normal projective varieties
    and $\tilde{Z}\to \tilde{X}$ is Galois.
    Further, the morphisms $\pi_Z:\tilde{Z}\to Z$ and 
    $\pi_Y:\tilde{Y}\to Y$ are birational.
    Then we have $\pi_{Z,\ast}\mathcal{O}_{\tilde{Z}}=\mathcal{O}_Z$
    and by the projection formula it suffices to show
    that $\pi^{\ast}_Z g^{\ast}f_{\ast}\mathcal{O}_Y$
    is the trivial bundle.

    Since $f$ is affine, we have
    \[
        \pi^{\ast}_Z g^{\ast}f_{\ast}\mathcal{O}_Y=
        \tilde{g}^{\ast} \pi^{\ast} f_{\ast}\mathcal{O}_Y =
        \tilde{g}^{\ast} \tilde{f}_{\ast} \mathcal{O}_{\tilde{Y}},
    \]
    see
    \cite[\href{https://stacks.math.columbia.edu/tag/02KG}{Lemma 02KG}]{sp}.
    Thus, we are reduced to showing the lemma in the case where $X$ 
    and $Y$ are normal projective varieties.
    We assume this from now on.

    Note that $g$ is an \'etale Galois cover satisfying 
    $Z\times_X Z \cong \bigsqcup_{\sigma \in \mathrm{Gal}(Z/X)} Z$.
    Then we conclude via affine base change that
    $g^{\ast}g_{\ast}\mathcal{O}_{Z} \cong 
    \bigoplus_{\sigma\in\mathrm{Gal}{Z/X}}\mathcal{O}_Z$.
    By adjunction we have $f_{\ast}\mathcal{O}_Y \subseteq g_{\ast}\mathcal{O}_{Z}$
    and conclude by flatness of $g$ that
    $g^{\ast}f_{\ast}\mathcal{O}_Y\subseteq g^{\ast}g_{\ast}\mathcal{O}_{Z}$.
    Thus, by the previous discussion
    we find that $g^{\ast}f_{\ast}\mathcal{O}_Y$ is a subsheaf of the trivial bundle.

    Fix a polarization $\mathcal{O}_X(1)$ of $X$ 
    and let $\mathcal{O}_Y(1)$ and $\mathcal{O}_Z(1)$ 
    the polarizations obtained via
    pullback to $Y$ and $Z$ respectively.
    We consider slope-(semi)stability with respect to these polarizations.
    
    Recall that slope-semistability is preserved under pushforward and
    pullback along an \'etale cover, by \cite[Lemma 3.2.1]{hl}. 
    Recall that the degree of a bundle does not change under
    pushforward along an \'etale cover,
    and the degree under pullback is multiplied by the degree of the cover.
    Thus, $g^{\ast}f_{\ast}\mathcal{O}_Y$
    slope-semistable of slope $0$ and
    contained in the trivial bundle by the above discussion.
    We conclude that $g^{\ast}f_{\ast}\mathcal{O}_Y$ is trivial.

    Since we are in characteristic $0$ this
    implies that $f_{\ast}\mathcal{O}_Y$ is polystable by
    \cite[Lemma 3.2.3]{hl}. Note that pullback along $g$ induces an 
    injection between cohomology groups of $X$ and $Z$.
    The naturality property of Chern classes then implies
    that all Chern classes of $f_{\ast}\mathcal{O}_Y$ vanish.
\end{proof}

Recall that $M^{P,G-ss}_X$ denotes the good moduli space of Gieseker
semistable Higgs bundles of rank $r$ on $X$ with Hilbert polynomial $P$,
and $P_0$ denotes the Hilbert polynomial of the trivial bundle of rank $r$.
The \emph{Dolbeault moduli space} is defined as the moduli space
$M^r_{X,\mathrm{Dol}}:=M^{G-ss,P_0}_X$. 
The $\kk$-points of $M^r_{X,\mathrm{Dol}}$
correspond to Gieseker polystable rank $r$ 
Higgs bundles with vanishing Chern classes.

We return to the situation where $s\in\mathscr{B}^r_X(\kk)$ 
is a spectral datum such that its discriminant $D(s)$ lies in
$H^0(X,\mathrm{Sym}^{r(r-1)}\Omega^1_X)\setminus W^1$.  

Recall from Lemma \ref{higgssheaf}
that $\pi'_{s*}\mathcal{O}_{X'_s}$
is equipped with a Higgs field $\theta'$
whose associated spectral datum is $s$.
It follows from Proposition \ref{polystability}
that the Higgs bundle $(\pi'_{s*}\mathcal{O}_{X'_s},\theta')$
corresponds to a point in the Dolbeault moduli space $M^r_{X,\mathrm{Dol}}$.

Thus, one might wonder when a stronger version
of Chen-Ng\^o's conjecture holds.

\begin{question}
When is the restricted Hitchin morphism
$h^r_{X,\mathrm{Dol}}:M^r_{X,\mathrm{Dol}}\to\mathscr{B}^r_X$ surjective?
\end{question}

\begin{remark}
    Note that for an irreducible spectral datum,
    that is, a $\kk$-point $s$ of $(r)\subset \mathscr{B}^r_X$,
    any torsion-free Higgs sheaf with spectral datum $s$
    is in fact slope-stable:
    it does not even have a non-trivial saturated Higgs-subsheaf
    since the characteristic polynomial is multiplicative
    in short exact sequences.

    However, direct sums of slope-stable sheaves need
    not be slope-polystable since the slope may vary
    and the spectral cover construction does not lend itself
    to controlling the slope (as far as we can see).
\end{remark}

In general, $h^r_{X,\mathrm{Dol}}$ is not surjective; 
see \cite[Example 3.4 (5)]{HL24} for a counterexample.
However, using the properness of $h^r_{X,\mathrm{Dol}}$
and assuming that $\mathscr{B}^r_X$ is irreducible,
we make the following observation, 
which partially answers the above question.

\begin{proposition}\label{strong}
    Let $X$ be a smooth projective scheme and let $r\in\mathbb{N}$
    be fixed.
    Suppose that $\mathscr{B}^r_X$ is irreducible,
    and that there is a point $s\in\mathscr{B}^r_X$ such that
    $D(s)$ lies in $H^0(X,\mathrm{Sym}^{r(r-1)}\Omega^1_X) \setminus W^1$.
    Then the map $h^r_{X,\mathrm{Dol}}:M^r_{X,\mathrm{Dol}}\to\mathscr{B}^r_X$ 
    is surjective.
\end{proposition}
\begin{proof}
    It follows from the hypothesis and from Lemma \ref{dense}
    that the subset $U$ of spectral data $s$ such that $D(s)$ 
    lands outside $W^1$ is a dense open subset of ${B}^r_X$.  

    For every point $s\in U$, the normalized spectral cover
    $\pi'_s:X'_s\to X$ is \'etale,
    and thus the associated Higgs bundle
    $(\pi'_{s*}\mathcal{O}_{X'_s},\theta')$
    is \'etale trivializable by Proposition \ref{polystability}
    and corresponds to a point in $M^r_{X,\mathrm{Dol}}$.
    In other words, $U$ is contained in $\mathrm{im}(h^r_{X,\mathrm{Dol}})$.

    By \cite[Theorem 6.11]{simp2},
    we know that $h^r_{X,\mathrm{Dol}}$ is proper.
    In particular, $\mathrm{im}(h^r_{X,\mathrm{Dol}})$ 
    is closed in $\mathscr{B}^r_X$.
    Since $U\subset\mathrm{im}(h^r_{X\mathrm{Dol}})$ and 
    $U$ is dense in $\mathscr{B}^r_X$,
    we conclude that $\mathrm{im}(h^r_{X,\mathrm{Dol}})$
    coincides with $\mathscr{B}^r_X$ set-theoretically.
\end{proof}

Similarly as for Lemma \ref{dense},
one can replace the irreducibility of $\mathscr{B}^r_X$
by the condition that every irreducible component of $\mathscr{B}^r_X$
contains a point $s$ such that $D(s)$ is not contained in $W^1$,
to arrive at the same conclusion.

Although Proposition \ref{strong} provides a sufficient condition for the
restricted Hitchin morphism $M^r_{X,\mathrm{Dol}}\to\mathscr{B}^r_X$ 
to be surjective, the irreducibility of
$\mathscr{B}^r_X$ is not a reasonable condition to check.
We can replace this
with the condition that the subsets $W^1$ be zero and still achieve the
surjectivity of $M^r_{X,\mathrm{Dol}}\to\mathscr{B}^r_X$.

\begin{proposition}\label{generalprop}
    Let $X$ be a smooth proper scheme and let $r\in\mathbb{N}$ be fixed.
    Suppose that $W^1=\{0\}\subset H^0(X,\mathrm{Sym}^{k(k-1)}\Omega^1_X)$
    for all $1\le k\le r$.
    Then for every $s\in\mathscr{B}^k_X$ the fiber $(h^k_X)^{-1}(s)$
    contains an \'etale trivializable Higgs bundle for all $1\le k\le r$.

    In particular, if $X$ is also projective, 
    the morphism $h^k_X:M^k_{X,\mathrm{Dol}}\to\mathscr{B}^k_X$ 
    is surjective for all $1\le k\le r$.
\end{proposition}
\begin{proof}
    For $s=0\in\mathscr{B}^k_X$ it is clear that the trivial bundle
    of rank $k$ is contained in the fiber $(h^k_X)^{-1}(0)$.
    Suppose $0\neq s\in\mathscr{B}^k_X$ is a $\kk$-point 
    such that $D(s)\neq0$.
    Then the normalized spectral cover $\pi'_s:X'_s\to X$
    is \'etale and $(\pi'_{s*}\mathcal{O}_{X'_s},\theta')$
    is an \'etale trivializable Higgs bundle for all $1\le k\le r$
    by Proposition \ref{polystability}.
    If $X$ is projective, this corresponds to a point
    in $M^k_{X,\mathrm{Dol}}$ for all $1\le k\le r$.
    
    Now, let $0\neq s\in\mathscr{B}^k_X$ be a point such that $D(s)=0$.
    Then by Lemma \ref{zerodisc} the characteristic polynomial
    associated with $s$ splits as a product $s=\prod_jt_j$,
    where $t_j\in(m_j)\subset\mathscr{B}^{m_j}_X$ 
    for some $m_j$ such that $\sum_jm_j=k$ and non-vanishing
    discriminant 
    $D(t_j)\in H^0(X,\mathrm{Sym}^{m_j(m_j-1)}\Omega^1_X)\setminus\{0\}$.

    By hypothesis, the normalized spectral cover 
    $\pi'_{t_j}:X'_{t_j}\to X$ is \'etale and the Higgs bundle
    $(\pi'_{t_j*}\mathcal{O}_{X'_{t_j}},\theta'_j)$ is \'etale
    trivializable for each $j$, again by Proposition \ref{polystability}.
    Then the direct sum $\bigoplus_j(\pi'_{t_j*}\mathcal{O}_{X'_{t_j}},\theta'_j)$ 
    is an \'etale trivializable Higgs bundle of rank $k$,
    and thus corresponds to a point in $M^k_{X,\mathrm{Dol}}$ if $X$ 
    is projective for all $1\le k\le r$.
\end{proof}

\begin{remark}\label{ettriv}
    From the proof of Proposition \ref{generalprop},
    we see that in the case where $X$ is projective,
    we have a stronger statement.
    That is, for each point $s\in\mathscr{B}^r_X$,
    the fiber $(h^r_{X,\mathrm{Dol}})^{-1}(s)$ contains a 
    Higgs bundle whose underlying vector bundle is \'etale trivializable. 
\end{remark}

The hypothesis in Proposition \ref{generalprop} 
is still somewhat intractable.
However, they are satisfied by $r$-small varieties.

\begin{theorem}\label{etaleconj}
    Let $X$ be a $r$-small variety,
    as in Definition \ref{rsmall} and let $r\in\mathbb{N}$ be fixed.
    Then $h^k_{X,\mathrm{Dol}}:M^k_{X,\mathrm{Dol}}\to\mathscr{B}^k_X$ 
    is surjective for all $1\le k\le r$.  
\end{theorem}
\begin{proof}
    It suffices to show that any $r$-small variety satisfies 
    the hypothesis of Proposition \ref{generalprop}.
    By definition, $\Omega^1_X$ is slope-semistable, 
    which implies that $\mathrm{Sym}^j\Omega^1_X$
    is slope-semistable for all $j\ge 1$ since
    we are working in characteristic $0$.

    We claim that 
    $W^1=\{0\}\subset H^0(X,\mathrm{Sym}^{k(k-1)}\Omega^1_X)$
    for all $1\le k\le r$.
    Indeed, suppose that there was a non-zero section 
    $\sigma\in H^0(X,\mathrm{Sym}^{k(k-1)}\Omega^1_X)$
    such that $\sigma$ vanishes along a divisor $E$ in $X$.
    Consider the injective morphism 
    $\mathcal{O}_X\hookrightarrow\mathrm{Sym}^{k(k-1)}\Omega^1_X$
    induced by $\sigma$.
    Denote the quotient by $Q:=\Sym^{k(k-1)}\Omega^1_X/\mathcal{O}_X$.
    Let $\mathcal{L}$ be the saturation of $\mathcal{O}_X$ in 
    $\Sym^{k(k-1)}\Omega^1_X$.
    Then the torsion subsheaf $T$ of $Q$ fits
    into a short exact sequence 
    $0\to \mathcal{O}_X\to\mathcal{L}\to T\to 0$.
    Since $\sigma$ vanishes along $E$ by assumption,
    we have that $T$ has support 
    $\Supp(T)=E$ and we conclude that $\mathcal{L}$
    has strictly positive degree, that is, $\deg(\mathcal{L})\geq 1$. 
    This contradicts the slope-semistability of $\Sym^{k(k-1)}\Omega^1_X$
    which is assumed to be slope-semistable of slope $<1$.
    
    We conclude that $W^1=\{0\}$ as claimed and that 
    $h^k_{X,\mathrm{Dol}}$ is surjective for all $1\le k\le r$ 
    by Proposition \ref{generalprop}.
\end{proof}

It is known that the cotangent bundle of a smooth projective variety with
numerically trivial canonical class is slope-semistable in characteristic zero (see, e.g. \cite[Theorem A]{G16}).
Thus, every $K_X$-numerically trivial variety is
$r$-small, and we arrive at the following special case.

\begin{theorem}\label{ktrivial}
Let $X$ be a smooth projective scheme with $K_X$-numerically trivial.
Then $h^r_{X,\mathrm{Dol}}:M^r_{X,\mathrm{Dol}}\to\mathscr{B}^r_X$
is surjective for all $r\in\mathbb{N}$.
\end{theorem}

Next, we make the following observation about the structure
of the spectral base for varieties which satisfy
the hypothesis in Proposition \ref{generalprop}.

\begin{proposition}\label{irred}
    Let $r\in\mathbb{N}$ be fixed
    and let $X$ be as in Proposition \ref{generalprop}.
    Then there are finitely many \'etale-trivializable vector bundles $\mathcal{V}_1,\dots,\mathcal{V}_m$ of rank $r$ on $X$ for some $m\in\mathbb{N}$ such that for each point $s\in\mathscr{B}^r_X$, the fiber $(h^r_X)^{-1}(s)$ contains some $(\mathcal{V}_i,\theta_i)\in M^r_{X,\mathrm{Dol}}$. 
\end{proposition}
\begin{proof}
    By induction on $r$, it suffices to prove the assertion for each point in the stratum $(r)\subset\mathscr{B}^r_X$. 
    
    By assumption, the discriminant of each point $s\in(r)$ has a vanishing locus of codimension $\ge2$. The corresponding normalized spectral cover $\pi'_s:X'_s\to X$ is thus \'etale of degree $r$ and by Proposition \ref{polystability}, the vector bundle $\pi'_{s*}\mathcal{O}_{X'_s}$ underlying the Higgs bundle $(\pi'_{s*}\mathcal{O}_{X'_s},\theta')$ is \'etale trivializable.
    
    Since there are finitely many \'etale covers of $X$ of degree $r$
    up to isomorphism, there are finitely many such \'etale trivializable vector bundles
    $\pi'_{s*}\mathcal{O}_{X'_s}$ of rank $r$ up to isomorphism underlying the Higgs bundles $(\pi'_{s*}\mathcal{O}_{X'_s},\theta')$, and we are done.
\end{proof}

\begin{remark}
With notation as in Proposition \ref{irred}, let $V_i$ denote the vector space of all Higgs fields on the vector bundle $\mathcal{V}_i$, for $1\le i\le m$. The Hitchin morphism induces maps $V_i\to\mathscr{B}^r_X$, $\theta\mapsto h^r_X(\mathcal{V}_i,\theta)$ for each $i$. Proposition \ref{irred} then implies that we can express the spectral base as a finite union $\mathscr{B}^r_X=\bigcup_{i=1}^mh^r_X(V_i)$. 

However, it is unclear whether this description would
imply irreducibility of the spectral base in this case. Rather, it may provide us with an example of the spectral base of a $K$-trivial variety where the spectral base is not irreducible.
\end{remark}

\bibliographystyle{plain}

\bibliography{bibliography}

\end{document}